\documentclass[reqno,11pt]{amsart}
\usepackage{eurosym}
\usepackage{amsfonts}
\usepackage{amssymb}
\usepackage{graphicx}
\usepackage{pstricks}
\usepackage{amsmath}
\usepackage{amsmath}
\usepackage{amsxtra}

\theoremstyle{plain}
\newtheorem{theorem}{Theorem}[section]
\newtheorem{lemma}[theorem]{Lemma}

\newtheorem{proposition}[theorem]{Proposition}
\newtheorem{corollary}[theorem]{Corollary}

\newtheorem{definition}[theorem]{Definition}

\newtheorem{problem}[theorem]{Problem}

\theoremstyle{definition}

\newtheorem{remark}[theorem]{Remark}

\numberwithin{equation}{section} 
\begin{document}
\title{ Joint spectra of spherical Aluthge transforms of commuting $n$-tuples of Hilbert space operators}
\author{Chafiq Benhida}
\address{UFR de Math\'{e}matiques, Universit\'{e} des Sciences et
Technologies de Lille, F-59655 Villeneuve-d'Ascq Cedex, France%
}
\email{chafiq.benhida@univ-lille.fr}
\author{Ra\'{u}l E. Curto}
\address{Department of Mathematics, The University of Iowa, Iowa City, Iowa
52242}
\email{raul-curto@uiowa.edu}
\author{Sang Hoon Lee}
\address{Department of Mathematics, Chungnam National University, Daejeon,
34134, Republic of Korea}
\email{slee@cnu.ac.kr}
\author{Jasang Yoon}
\address{School of Mathematical and Statistical Sciences, The University of
Texas Rio Grande Valley, Edinburg, Texas 78539, USA}
\email{jasang.yoon@utrgv.edu}
\thanks{The first named author was partially supported by Labex CEMPI (ANR-11-LABX-0007-01).}
\thanks{The second named author was partially supported by NSF Grant
DMS-1302666.}
\thanks{The third named author was partially supported by NRF
(Korea) grant No. 2016R1D1A1B03933776.}
\thanks{The fourth named author was partially supported by a grant from the
University of Texas System and the Consejo Nacional de Ciencia y Tecnolog%
\'{\i}a de M\'{e}xico (CONACYT)}
\subjclass[2000]{Primary 47B20, 47B37, 47A13, 28A50; Secondary 44A60, 47-04,
47A20}
\keywords{spherical Aluthge transform, Taylor spectrum, Taylor essential spectrum, Fredholm pairs, Fredholm index}


\begin{abstract}
Let $\mathbf{T} \equiv (T_1,\cdots,T_n)$ be a commuting $n$-tuple of operators on a Hilbert space $\mathcal{H}$, and let $T_i \equiv V_i P \; (1 \le i \le n)$ be its canonical joint polar decomposition (i.e., $P:=\sqrt{T_1^*T_1+\cdots+T_n^*T_n}$, $(V_1,\cdots,V_n)$ a joint partial isometry, and $\bigcap_{i=1}^n \ker T_i = \bigcap_{i=1}^n \ker V_i = \ker P)$. \ The spherical Aluthge transform of $\mathbf{T}$ is the (necessarily commuting) $n$-tuple $\hat{\mathbf{T}}:=(\sqrt{P}V_1\sqrt{P},\cdots,\sqrt{P}V_n\sqrt{P})$. \ We prove that $\sigma_T(\hat{\mathbf{T}})=\sigma_T(\mathbf{T})$, where $\sigma_T$ denotes Taylor spectrum. \ We do this in two stages: away from the origin we use tools and techniques from criss-cross commutativity; at the origin we show that the left invertibility of $\mathbf{T}$ or $\hat{\mathbf{T}}$ implies the invertibility of $P$. \ As a consequence, we can readily extend our main result to other spectral systems that rely on the Koszul complex for their definition.

\vskip 0.5\baselineskip

{\it To cite this article: C. Benhida, R.E. Curto, S.H. Lee, J. Yoon, C. R. Acad. Sci. Paris, Ser. I 340 (2005).}

\vskip 0.5\baselineskip

{\bf R\'esum\'e}. \ Soit $\mathbf{T} \equiv (T_1,\cdots,T_n)$ un n-uplet commutatif d'opérateurs sur un espace de Hilbert $\mathcal{H}$, et soit $T_i \equiv V_i P \; (1 \le i \le n)$  sa décomposition polaire jointe canonique  (i.e., $P:=\sqrt{T_1^*T_1+\cdots+T_n^*T_n}$, $(V_1,\cdots,V_n)$  est une isométrie partielle jointe, et $\bigcap_{i=1}^n \ker T_i = \bigcap_{i=1}^n \ker V_i = \ker P)$. \  La transformée d'Aluthge sphérique de $\mathbf{T}$  est le $n$-uplet (nécessairement commutatif)  $\hat{\mathbf{T}}:=(\sqrt{P}V_1\sqrt{P},\cdots,\sqrt{P}V_n\sqrt{P})$. \  Nous démontrons que $\sigma_T(\hat{\mathbf{T}})=\sigma_T(\mathbf{T})$, où $\sigma_T$ d\'esigne le spectre de Taylor. \ Nous procédons pour cela en deux étapes:  En dehors de l'origine nous utilisons les outils et les techniques de la commutativité criss-cross; à l'origine nous prouvons que l'inversibilité  à gauche de $\mathbf{T}$ ou de  $\hat{\mathbf{T}}$ implique  l'inversibilité de $P$. \ Comme conséquence, nous pouvons étendre notre résultat à d'autres syst\`emes spectraux définis à partir des complexes de Koszul. 

\vskip 0.5\baselineskip

{\it Pour citer cet article~: C. Benhida, R.E. Curto, S.H. Lee, J. Yoon, C. R. Acad. Sci. Paris, Ser. I 340 (2005).}

\end{abstract}

\maketitle


\section{Introduction}
\label{Int}
Let $\mathcal{H}$ be a complex infinite dimensional Hilbert space, let $\mathcal{B}(\mathcal{H})$ denote the algebra of bounded linear operators on $\mathcal{H}$, and let $T \in \mathcal{B}(\mathcal{H})$. \ For $T \equiv V|T|$ the canonical polar decomposition of $T$, we let $\tilde{T}:=|T|^{1/2}V|T|^{1/2}$ denote the Aluthge transform of $T$ \cite{Alu}. \ It is well known that $T$ is invertible if and only if $\tilde{T}$ is invertible; moreover, the spectra of $T$ and $\tilde{T}$ are equal. \ Over the last two decades, considerable attention has been given to the study of the Aluthge transform; cf. \cite{Ando}--\cite{BRZ}, \cite{CJL}, \cite{CuYo1}--\cite{LLY2}, \cite{Rion}, \cite{W}--\cite{Yam}). \ Moreover, the Aluthge transform has been generalized to the case of powers of $|T|$ different from $\frac{1}{2}$ (\cite{Ben}, \cite{BCKL1}, \cite{Cha}, \cite{LLY}) and to the case of commuting pairs of operators (\cite{CuYo1}, \cite{CuYo2}).

In this note, we focus on the spherical Aluthge transform \cite{CuYo2}. \ Although our results hold for arbitrary $n>2$, for the reader's convenience we will focus on the case $n=2$; that is, the case of commuting pairs of Hilbert space operators. \ Let $\mathbf{T} \equiv (T_1,T_2)$ be a commuting pair of operators on $\mathcal{H}$. \ We now consider the canonical polar decomposition of the column operator 
$\left(
\begin{array}{c}
T_{1} \\
T_{2}%
\end{array}%
\right)$; that is,
$\left(
\begin{array}{c}
T_{1} \\
T_{2}%
\end{array}%
\right)
\equiv \left(
\begin{array}{c}
V_{1} \\
V_{2}%
\end{array}%
\right) P$,
where $P:=\sqrt{T_1^*T_1+T_2^*T_2}$ and $\left(
\begin{array}{c}
V_{1} \\
V_{2}%
\end{array}%
\right)$ is a (joint) partial isometry, and subject to the constraint $\bigcap_{i=1}^2 \ker T_i = \bigcap_{i=1}^2 \ker V_i = \ker P$. \ 

The {\it spherical Aluthge transform} of $\mathbf{T}$ is the (necessarily commuting) $n$-tuple 
\begin{equation} \label{sphAlu}
\widehat{\mathbf{T}}:=(\sqrt{P}V_1\sqrt{P},\cdots,\sqrt{P}V_n\sqrt{P}) \; \; \; \textrm{(\cite{CuYo1}, \cite{CuYo2})}.
\end{equation}

For a commuting pair $\mathbf{T} \equiv (T_1,T_2)$ of operators on $\mathcal{H}$, the Koszul complex associated with $\mathbf{T}$ is given as
\begin{equation*}
K(\mathbf{T,}\mathcal{H}):0\overset{0}{\rightarrow }\mathcal{H}\overset{\left(
\begin{array}{c}
T_{1} \\
T_{2}%
\end{array}%
\right)}{\longrightarrow }\mathcal{H}\oplus \mathcal{H}
\overset{\left(-T_2 \; T_1 \right)}{\longrightarrow }\mathcal{H}\overset{0}{\longrightarrow }0.
\end{equation*}

\begin{definition}
A commuting pair $\mathbf{T}$ is said to be \textit{(Taylor) invertible} if its associated Koszul
complex $K(\mathbf{T,}\mathcal{H})$ is exact. \ The Taylor spectrum of $\mathbf{T}$ is 
\begin{equation*}
\begin{tabular}{l}
$\sigma _{T}(\mathbf{T}):=\left\{ (\lambda _{1},\lambda _{2})\in \mathbb{C}%
^{2}:K\left( \left( T_{1}-\lambda _{1},T_{2}-\lambda _{2}\right) \mathbf{,}%
\text{ }\mathcal{H}\right) \text{ is not invertible}\right\} $.%
\end{tabular}%
\end{equation*}
The pair $\mathbf{T}$ is called \textit{Fredholm} if each map in the Koszul complex $K(\mathbf{T},\mathcal{H})$ has closed range and all the homology quotients are finite-dimensional. \ The Taylor essential spectrum is
\begin{equation*}
\begin{tabular}{l}
$\sigma _{Te}(\mathbf{T}):=\left\{ (\lambda _{1},\lambda _{2})\in \mathbb{C}%
^{2}:\left( T_{1}-\lambda _{1},T_{2}-\lambda _{2}\right) \text{\textbf{\ }is
not Fredholm}\right\} $.%
\end{tabular}%
\end{equation*}
\end{definition}

J.L. Taylor showed in \cite{Tay1} and \cite{Tay2} that, if $\mathcal{H}\neq \{0\}$, then $\sigma _{T}(%
\mathbf{T})$ is a nonempty, compact subset of the polydisc of multiradius $r(%
\mathbf{T}):=(r(T_{1}),r(T_{2})),$ where $r(T_{i})$ is the spectral radius
of $T_{i}$ \ ($i=1,2$). \ (For additional facts
about these joint spectra, the reader is referred to \cite{Cu1}--\cite{Cu3} and \cite{Vas}.)

As shown in \cite{Cu2} and \cite{Appl}, the Fredholmness of $\mathbf{T}$ can be detected in the Calkin algebra $\mathcal{Q}(\mathcal{H}):=\mathcal{B}(\mathcal{H})/\mathcal{K}(\mathcal{H})$. \ (Here $\mathcal{K}$ denotes the closed two--sided ideal of compact operators; we also let $\pi:\mathcal{B}(\mathcal{H}) \longrightarrow \mathcal{Q}(\mathcal{H})$ denote the quotient map.) \ Concretely, $\mathbf{T}$ is Fredholm on $\mathcal{H}$ if and only if the pair of left multiplication operators $L_{\pi(\mathbf{T})}:=(L_{\pi(T_1)},L_{\pi(T_2)})$ is Taylor invertible when acting on $\mathcal{Q}(\mathcal{H})$. \ In particular, $\mathbf{T}$ is left Fredholm on $\mathcal{H}$ if and only if $L_{\pi(\mathbf{T})}$ is bounded below on $\mathcal{Q}(\mathcal{H})$.

\begin{problem}
\label{Problem 5} Let\textbf{\ }$\mathbf{T}\equiv (T_{1},T_{2})$ be a
commuting pair of operators.\newline
(i) Assume that $\mathbf{T}$ be (Taylor) invertible (resp. Fredholm). \ Is $\widehat{\mathbf{T}}$
also (Taylor) invertible (resp. Fredholm)?\newline
(ii) \ Is the Taylor spectrum (resp. Taylor essential spectrum) of $%
\widehat{\mathbf{T}}$ equal to that of $\mathbf{T}$?
\end{problem}

We first prove that $\sigma_T(\widehat{\mathbf{T}})=\sigma_T(\mathbf{T})$. \ We do this in two stages: away from the origin we use tools and techniques from criss-cross commutativity; at the origin we show that the left invertibility of $\mathbf{T}$ or $\widehat{\mathbf{T}}$ implies the invertibility of $P$; $P$ then helps establish an isomorphism between the relevant Koszul complexes. \ As a consequence, we can readily extend the above result to other spectral systems that rely on the Koszul complex for their definition, including spectral systems on $\mathcal{Q}(\mathcal{H})$.

\section{Main Results}
\label{MainResults}

Recall the joint polar decomposition of $\mathbf{T}$ and the spherical Aluthge transform of $\mathbf{T}$; cf. (\ref{sphAlu}). \ We now state our first main result. 

\begin{theorem}
\label{basic} Assume that $\mathbf{T}$ or $\widehat{\mathbf{T}}$ is left invertible; that is, the associated Koszul complex is exact at the left stage, and the range of the corresponding boundary map is closed. \ Then the operator $P$ is invertible.
\end{theorem}

\begin{proof}
\noindent
{\bf Case 1}. \ If $\mathbf{T}$ is left invertible, then $T_1^*T_1+T_2^*T_2$ is invertible, and therefore $P$ is invertible.

\noindent {\bf Case 2}. \ If $\widehat{\mathbf{T}}$ is left invertible, then it is bounded below; that is, there exists a constant $c>0$ such that
$$
\left\|\sqrt{P}V_1\sqrt{P}x\right\|^2+\left\|\sqrt{P}V_2\sqrt{P}x\right\|^2\ge c^2 \left\|x\right\|^2.
$$
Since $(V_1,V_2)$ is a joint partial isometry, it readily follows that 
$$
\left\|\sqrt{P}x\right\|^2+\left\|\sqrt{P}x\right\|^2\ge \frac{c^2}{\left\|P\right\|} \left\|x\right\|^2.
$$
As a result, $\sqrt{P}$ is bounded below, so $P$ is invertible.
\end{proof}

We are now ready to state our second main result. 

\begin{theorem}
\label{Theorem 7} Let $\mathbf{T}=\mathbf{(}T_{1},T_{2})$ be a
commuting pair of operators on $\mathcal{H}$. \ Then
$$
\mathbf{T} \textrm{ is (Taylor) invertible } \Longleftrightarrow \widehat{\mathbf{T}} \textrm{ is (Taylor) invertible}.
$$
\end{theorem}

We now recall the notion of criss-cross commutativity.

\begin{definition} \label{criss}
Let $\mathbf{A} \equiv (A_1,\cdots,A_n)$ and $\mathbf{B} \equiv (B_1,\cdots,B_n)$ be two $n$-tuples  of operators on $\mathcal{H}$. \ We say that $\mathbf{A}$ and $\mathbf{B}$ criss-cross commute (or that $\mathbf{A}$ criss-cross commutes with $\mathbf{B}$) if $A_iB_jA_k=A_kB_jA_i$ and $B_iA_jB_k=B_kA_jB_i$ for all $i,j,k=1,\cdots,n$. \ Observe that we do not assume that $\mathbf{A}$ or $\mathbf{B}$ is commuting.
\end{definition}

\begin{definition}
Given two $n$-tuples $\mathbf{A}$ and $\mathbf{B}$ we define $\mathbf{AB}:=(A_1B_1,\cdots,A_nB_n)$ and $\mathbf{BA}:=(B_1A_1,\cdots,B_nA_n)$.
\end{definition}

\begin{remark}
It is an easy consequence of Definition \ref{criss} that, if $\mathbf{A}$ and $\mathbf{B}$ criss-cross commute and $\mathbf{AB}$ is commuting, then $\mathbf{BA}$ is also commuting.
\end{remark}

\begin{lemma} Let $\mathbf{T}\equiv(T_1,T_2)$ be a commuting pair of operators on $\mathcal{H}$, let $P:=\sqrt{T_1^*T_1+T_2^*T_2}$, and let $\widehat{\mathbf{T}}$ be its spherical Aluthge transform. \ Then $\mathbf{A} \equiv (A_1,A_2):=(\sqrt{P},\sqrt{P})$ and $\mathbf{B} \equiv (B_1,B_2):=(V_1\sqrt{P},V_2\sqrt{P})$ criss-cross commute. \ As a consequence, $\widehat{\mathbf{T}} (=\mathbf{B}\mathbf{A})$ is commuting.  
\end{lemma}

\begin{lemma} \label{criss thm} \ (cf. \cite{BenZe1} and \cite{BenZe2}) \ Let $\mathbf{A}$ criss-cross commute with $\mathbf{B}$ on $\mathcal{H}$, and assume that $\mathbf{AB}$ is commuting. \ Then $\sigma_T(\mathbf{BA}) \setminus \{\mathbf{0}\}=\sigma_T(\mathbf{AB}) \setminus \{\mathbf{0}\}$.
\end{lemma}

We now prove our third main result.

\begin{theorem} \label{Theorem 8} Let $\mathbf{T}=\mathbf{(}T_{1},T_{2})$ be a
commuting pair of operators on $\mathcal{H}$. \ Then
$$
\sigma_T(\mathbf{T})=\sigma_T(\widehat{\mathbf{T}}).
$$
\end{theorem}

\begin{proof}
Let $\lambda \in \mathbb{C}^2$. \ If $\lambda =(0,0)$, use Theorem \ref{Theorem 7}; if $\lambda \ne (0,0)$, use Lemma \ref{criss thm}.
\end{proof}

\begin{remark}
(i) \ Theorems \ref{basic}, \ref{Theorem 7} and \ref{Theorem 8} can be easily extended to other spectral systems whose definition is given in terms of the Koszul complex; e.g., the left $k$-spectral systems $\sigma_{\pi,k}$ defined by W. S\l odkowski and W. \.{Z}elazko (\cite{Slo}, \cite{SZ}). \ For, the Proof of Theorem \ref{basic} (which uses only left invertibility of the relevant Koszul complex) works well in case $\mathbf{T}$ or $\widehat{\mathbf{T}}$. \ Once we know that $\sqrt{P}$ is invertible, the Koszul complexes of $\mathbf{T}$ and $\widehat{\mathbf{T}}$ are isomorphic, so $0 \notin \sigma_{\pi,k}(\mathbf{T})$ if and only if $0 \notin \sigma_{\pi,k}(\widehat{\mathbf{T}})$. \newline
(ii) \ Similarly, Theorem \ref{criss thm} admits an easy extension to S\l odkowski's left $k$-spectra (cf. \cite{BenZe1}, \cite{BenZe2}), since the Proof of Theorem \ref{criss thm} relies on the isomorphism of the Koszul complexes for $\mathbf{T}$ and $\widehat{\mathbf{T}}$, implemented by $\sqrt{P}$. \newline
(iii) \ On the other hand, the above results cannot be extended to S\l odkowski's right $k$-spectra; for, consider the adjoint $U_+^*$ of the (unweighted) unilateral shift $U_+$. \ It is easy to see that $U_+^*$ is onto while $\widehat{U_+^*}$ is not.
\end{remark}

Our final main result deals with Fredholmness.

\begin{theorem} \label{Theorem 9} Let $\mathbf{T}=\mathbf{(}T_{1},T_{2})$ be a
commuting pair of operators on $\mathcal{H}$. \ Then
$$
\sigma_{Te}(\mathbf{T})=\sigma_{Te}(\widehat{\mathbf{T}}).
$$
Moreover, for each $\lambda \notin \sigma_{Te}(\mathbf{T})$ we have 
$$
\textrm{ind }(\mathbf{T}-\lambda)=\textrm{ind }(\widehat{\mathbf{T}}-\lambda),
$$
where $\textrm{ind}$ denotes the Fredholm index. 
\end{theorem}

\begin{proof}[Sketch of Proof]
In Theorem \ref{basic} one can replace ``left invertible" for the Koszul complex with ``left Fredholm" and ``invertible" for $P$ with ``Fredholm." \ A similar adjustment works for Theorems \ref{Theorem 7} and \ref{Theorem 8}. \ In the analog of Theorem \ref{Theorem 7} one first proves that $\sqrt{P}$ is bounded below in the orthogonal complement of $\ker T_1 \cap \ker T_2$; since this kernel is finite dimensional, it follows that $\sqrt{P}$ is Fredholm. \ In Theorem \ref{Theorem 8} one needs to replace Lemma \ref{criss thm} with the results for Fredholmness proved in \cite{BenZe1}, \cite{BenZe2}, \cite{Li1} and \cite{Li2}. \ While Li's results only guarantee that $\textrm{ind }(\mathbf{T}-\lambda)=\textrm{ind }(\widehat{\mathbf{T}}-\lambda)$ whenever $\lambda \ne (0,0)$, the continuity of the Fredholm index (cf. \cite{Cu2}) does the rest.
\end{proof}


\end{document}